\documentclass[a4paper,11pt]{elsarticle}
\usepackage{amsthm}
\usepackage{amsmath}
\usepackage{amssymb}
\usepackage{enumerate}
\usepackage{graphics}
\usepackage{subfig}
\usepackage{tikz}
\usetikzlibrary{arrows}
\def\SD{\mathbf{dist}}
\def\bd{\mathbf{d}}
\def\bbd{\mathbf{bd}}
\def\bdd{\mathbf{dd}}
\def\mc{\mathbf{maxC}}
\def\a{d^+}
\def\b{d^-}
\def\BrbG{balanced red-blue graph}
\newcommand\old[1]{}

\makeatletter
\def\ps@pprintTitle{%
  \let\@oddhead\@empty
  \let\@evenhead\@empty
  \def\@oddfoot{\reset@font\hfil\thepage\hfil}
  \let\@evenfoot\@oddfoot

\makeatother
    }
\numberwithin{equation}{section}
\newtheorem{theorem}{Theorem}[section]
\newtheorem{prop}[theorem]{Proposition}
\newtheorem{conj}{Conjecture}
\newtheorem{lemma}[theorem]{Lemma}

\theoremstyle{definition}
\newtheorem{definition}[theorem]{Definition}
\theoremstyle{remark}
\newtheorem{remark}[theorem]{Remark}

\begin{document}
\begin{frontmatter}
\title{On the swap-distances of different realizations of a graphical degree sequence\tnoteref{darpa}}
\author[renyi]{P\'eter L. Erd\H os\fnref{elp}}
\author[ttk]{Zolt\'an Kir\'aly\fnref{kir}}
\author[renyi]{Istv\'an Mikl\'os\fnref{mik}}
\address[renyi]{Alfr\'ed R{\'e}nyi Institute, Re\'altanoda u 13-15 Budapest, 1053 Hungary\\
        {\tt email}: $<$erdos.peter,miklos.istvan$>$@renyi.mta.hu}
\address[ttk]{Department of Computer Science and EGRES (MTA-ELTE), E\"otv\"os University, P\'azm\'any P\'eter s\'et\'any 1/C, Budapest, 1117 Hungary\\
        {\tt email}: kiraly@cs.elte.hu}
\fntext[elp]{PLE   was supported in part by the Hungarian NSF, under contract NK 78439.}
\fntext[kir]{ZK was supported by
grants (no.\ CNK 77780 and no.\ CK 80124) from the National Development
 Agency of Hungary, based on a source from the Research and Technology
Innovation Fund, and also by T\'AMOP grant 4.2.1./B-09/1/KMR-2010-0003.}
\fntext[mik]{IM was supported in part by a Bolyai postdoctoral stipend and by the Hungarian NSF, under contract F61730.}
\tnotetext[darpa]{PLE and IM acknowledge financial support from grant \#FA9550-12-1-0405 from the U.S. Air Force Office of Scientific Research (AFOSR) and the Defense Advanced
Research Projects Agency (DARPA).}
\begin{abstract}
One of the first graph theoretical problems which got serious attention
(already in the fifties of the last century)  was to decide whether a given
integer sequence is equal to the degree sequence of a simple graph (or it is
{\em graphical} for short). One method to solve this problem is the greedy
algorithm of Havel and Hakimi, which is based on the {\em swap}
operation. Another, closely related question is to find a sequence of swap
operations to transform one graphical realization into another one of the same
degree sequence. This latter problem got particular emphases in connection of
rapidly mixing Markov chain approaches to sample uniformly  all possible
realizations of a given degree sequence. (This becomes a matter of interest in
connection of -- among others -- the study of large social networks.) Earlier
there were only crude upper bounds on the shortest possible length of such
swap sequences between two realizations. In this paper we develop formulae
(Gallai-type identities) for these {\em swap-distance}s of any two realizations of simple undirected or directed degree sequences. These identities improves considerably the known upper bounds on the swap-distances.\\
{\bf AMS classification (2010)}: Primary - 05C07,  Secondary - 05C20, 05C45
\end{abstract}
\end{frontmatter}

\section{Introduction}\label{sec:intro}

The comprehensive study of graphs (or more precisely the {\em linear graphs}, as it was called in that time) began sometimes in the late forties, through seminal works by P.\ Erd\H{o}s, P.\ Tur\'an, W.T.\ Tutte, T.\ Gallai and others. One problem which received considerable attention was the existence of certain subgraphs of a given graph $G.$ Such a subgraph could be, for example,  a perfect matching in a (not necessarily bipartite) graph, or a Hamiltonian cycle, etc.  Generally these substructures are called {\em factors}. The first
couple of important and rather general results of this kind were due to Tutte (in 1952) who gave necessary and sufficient conditions for the existence of  $f$-factors \cite{T52,T54}.

In cases where $G$ is a complete graph, the $f$-factor problem becomes easier:
then we are simply interested in the existence of a graph with a given degree
sequence,  and at least two solutions of different kind were developed around
1960. One was due to Havel \cite{H55} who constructed a famous greedy
algorithm to answer this {\em degree sequence problem}. His algorithm was
based on the notion of {\em swap}. It is interesting to mention the almost
completely forgotten  paper of Senior (\cite{S51}) who studied the problem of
generating graphs with multiple edges but without loops: his goal was to find
possible molecules with given composition but with different structures. He
also discovered the swap operation, but he called it {\em transfusion}. The
other approach was the equally famous Erd\H{o}s-Gallai theorem (\cite{EG})
which gave a necessary and sufficient condition in the form of a sequence of
inequalities. In this latter paper  Havel's method was an ingredient of the
proof and the authors also observed that their result is a consequence of
Tutte's $f$-factor theorem.

In 1962 Hakimi studied the degree sequence problem in undirected graphs with multiple edges and loops (\cite{H62}). He developed an Erd\H{o}s-Gallai type result for this much simpler case, and for the case of simple graphs he rediscovered the greedy algorithm of Havel. Since then this algorithm is referred to as the {\em Havel--Hakimi algorithm}.

Already from the general $f$-factor theorem of Tutte one can derive
a polynomial time algorithm to solve the degree sequence problem,
but it was not done that time. Havel's algorithm provided a quadratic (in $n$)
time construction method of the required graphical {\em realization}.

\medskip\noindent
The construction of -- preferable "typical" -- graphical realizations of given degree sequences became an important  problem in the last two decades in connection of the emergence of huge networks in social sciences, medicine, biology or the internet technology, naming only some.

Mentioning just one example here, data is collected from anonymous surveys  in epidemics studies of sexually transmitted diseases, where the individuals specify the number of different partners they have had in a given period of time, without revealing their identity. In this case, epidemiologists should construct typical contact graphs obeying the empirical degree sequence to estimate epidemiological parameters.

To construct all possible realizations of a given degree sequence is typically very time consuming task since usually there are exponentially many different realizations. Here we do not consider the computationally almost hopeless isomorphism problem. In this paper the vertices are labeled (therefore distinguishable) and two isomorphic realizations where the isomorphism changes the labels are considered to be different ones.

The methodology to construct all possible realizations is already not self-evident: for example the Havel-Hakimi algorithm is not strong enough to find all of them. It is also important that no particular realization should be outputted more than once, and, finally, that the waiting times between two consecutive outputs  should not be too long. These concerns were addressed in \cite{kiraly}.

However when our goal is to find a "typical" realization but there are exponentially many different ones then generating all of them and choose one realization randomly is simple not feasible. One way to overcome this hardness is to construct a good Monte Carlo Markov Chain (MCMC) method. To that end we need a particular operation to walk on the space of the different realizations; this operation is called a {\em swap}, and it is essentially the same as the operation in Havel's algorithm: we choose four vertices, where in the induced subgraph there is a one-factor, while an other one-factor is missing, and we exchange the existing one-factor into the missing one. This clearly preserves the degree sequence.

It is interesting to recognize, that -- as one can learn this fact from Erd\H{o}s and Gallai -- the swap operation for this purpose was originally discovered by Petersen, already in 1891 (\cite{pet}).

The problem of Erd\H{o}s-Gallai type characterization of bipartite degree sequences were studied already in the mid-fifties by Gale (in \cite{G57}) using network flow techniques.   In the same year Ryser gave a direct proof for this characterization, using a matrix theoretical language (\cite{R57}) and showed  that  any particular realization of a bipartite degree sequence can be transformed, using sequence of swaps, into any other realization.  Both results were formulated on the language of directed graphs without multiple edges but with possible loops. (For the connection between bipartite and directed degree sequence problems see Section \ref{sec:dir}.) The corresponding result for simple directed graphs (no loops, no multiple edges) is due to Fulkerson (\cite{F60}).

Havel-Hakimi type results are part of the folklore in connection with bipartite degree sequences but it is hard to find a definitive reference for it (but book \cite{W} discusses the problem in Exercise 1.4.32 and paper \cite{K09} provides one proof as a by-product).

In case of simple directed graphs Havel-Hakimi type results were proved by
Kleitman and Wang (\cite{KW73}) for an extension given by Kundu \cite{Ku73}.
Swap sequences between realizations of directed graphs  were rediscovered in \cite{EMT09}. The situation here is more complicated than in the previous cases: using only two edges for a swap is not always enough, sometimes we have to use three-edge swaps. Moreover there are two different kinds of three-edge swaps: in the {\bf type 1} swap the three edges form an oriented $C_3$ and the result of the swap is the oppositely  oriented triangle. In the {\bf type 2} swap the three involved {\em directed edges} determine  four vertices (see \cite{KW73, EMT09}).

In \cite{EMT09} a weak upper bound was proved for the swap-distance of two realizations of directed degree sequences. In the proof all three types of swaps were applied, and counted as one.  However, as LaMar proved recently (in {\cite{lamar}), swap sequences between any two realizations may omit completely type 2 triple-swaps. In Section \ref{sec:dir} we will strengthen this result (see Theorem \ref{lm:lamar}). Finally Greenhill proved (\cite{green}) that in case of regular directed degree sequences (when all in-degrees and out-degrees are the same) all triple-swaps can be omitted (if $n>3$). The reason for the need for some type 2 swaps in \cite{KW73, EMT09} was simple: in the Havel-Hakimi situation (by analogy) it was sought for one swap changing a particular directed edge to a particular non-edge. However, when a  transformation sequence is looked for, then this is not a requirement. As it turned out, type 2 triple-swaps are so called non-triangular $C_6$-swaps (see Remark \ref{lm:no-trian}) and can be substituted by two regular swaps.

These problems have long and lively history but we do not survey that here. We just want to point out that knowing the maximum length of the necessary swap sequences can lead to better estimations for the mixing time of Markov chains using swap operations.
Until now there were only weak upper bounds on those lengths: they are surely shorter, than twice the number of edges in the realizations (which is equal to the sum of the values in the degrees sequence). This applies for simple directed or undirected graphical degree sequences (including bipartite ones as well).  (See for example \cite{EMT09}).

The main goal of this paper is to determine a formula for the {\em swap-distance} (that is the length of the possible shortest swap sequence) between any two particular realizations $G_1$ and $G_2.$ Here we will prove a Gallai-type identity
\begin{equation}\label{eq:weak}
\SD(G_1, G_2) = \frac{ |E(G_1) \Delta E(G_2)|}{2} -\mc(G_1,G_2),
\end{equation}
where $\Delta$ denotes the symmetric difference and $\mc$ is a positive integer which depends on the realizations. In case of directed degree sequences triple-swaps should count as 2 in the swap-distance. (For an explanation, see the definitions after Lemma \ref {lm:no-touch}.) However, as it will turn out, we only need type 1 triple-swaps. We can  forbid type 2 triple-swaps while the equation does not change.

It is very important to understand that while the right side of equation (\ref{eq:weak}) can be interpreted indeed as "the exact value of the swap-distance" -- actually $\mc$ is possibly (probably) not an efficiently computable value. We think that the right goal here is to find good estimations for $\mc$. We made the first steps into this direction, see Theorems \ref{th:undir_bound}, \ref{th:bip_bound} and \ref{th:dir_bound}.

\bigskip\noindent

The structure of the paper is the following: in Section \ref{sec:def} we introduce the definitions and recall some known facts and algorithms. In Section \ref{sec:undir} we prove (\ref{eq:weak}) for undirected degree sequences. In the very short Section \ref{sec:bi} we describe the consequences of the previous results for bipartite degree sequences. Finally in Section \ref{sec:dir} we discuss the problem for directed degree sequences based on further considerations on realizations of bipartite degree sequences.

\section{Definitions, notations}\label{sec:def}

Throughout the paper $G$ denotes an undirected simple graph with vertex set $V(G)= \{v_1,v_2,\ldots,v_n \}$ and edge set $E(G)$. Consider a  sequence of positive integers $\bd=(d_1,d_2,\ldots,d_n).$ If there is a simple graph $G$ with degree sequence $\bd$, i.e., where for each $i$ we have $d(v_i)=d_i$, then we call the sequence $\bd$ a {\em graphical sequence} and in this case we also say that $G$ {\em realizes}  $\bd$.

\bigskip\noindent

The analogous notions for bipartite graphs are the following: if $B$ is a simple bipartite graph then its vertex classes will be denoted by $U(B)=\{ u_1,\ldots, u_k\}$ and $ W(B)=\{w_1,\ldots,w_\ell\}$, and we keep the notation $V(B)=U(B)\cup W(B)$.  The {\em
bipartite degree sequence} of $B$, $\bbd(B)$ is defined as follows:
$$
\bbd(B)=\Big(\big (d(u_1), \ldots,d(u_k)\bigr ), \bigl (d(w_1),\ldots,d(w_\ell)\bigr )\Big ).
$$

Let $G$ be a simple graph and assume that $a,b,c$ and $d$ are different vertices.  If $G$ is bipartite graph $B$ then we also require that for  $a,b\in U(B),\; c,d\in W(B)$. Furthermore assume that $(a,c), (b,d) \in E(G)$ while $(b,c), (a,d) \not \in E(G)$. Then
\begin{equation}\label{eq:swap}
E(G')= E(G) \setminus \{(a,c), (b,d)\} \cup \{(b,c), (a,d)\}
\end{equation}
is another realization of the same degree sequence (and if $G$ is a bipartite
graph then $G'$ remains bipartite). The operation described above is called a
{\em swap}. This operation is used in the Havel-Hakimi algorithm, and Petersen
proved \cite{pet} -- and several authors later reproved -- that any
realization of a degree sequence can be transformed into any another
realization of the same degree sequence using only consecutive swap
operations.

As throughout the paper all graphs will be simple, from this point we will omit the word ``simple''.

A graph $G$, where the edges are colored by either red or blue, will be called a {\em red-blue graph}. For vertex $v$ denote by $d_r(v)$ and $d_b(v)$ the degree of vertex $v$ in red and blue edges, resp. This red-blue graph is {\em balanced} if for each $v\in V(G)$ we have $d_r(v)=d_b(v).$

A \emph{circuit} in a graph $G$ is a closed trail (each edge can be used at most once).
As the graph is simple, a circuit is determined by the sequence of the vertices $v_0,\ldots,v_t$, where $v_0=v_t$. Note that there can also be other indices $i<j$ such that
$v_i=v_j$. A circuit is called a \emph{cycle}, if its simple, i.e., for any $i<j$, $\;v_i=v_j$ only if $i=0$ and $j=t$.

A circuit (or a cycle) in a \BrbG\ is called \emph{alternating}, if the color of its edges alternates (i.e., the color of the edge from $v_i$ to $v_{i+1}$ differs from the color of the edge from $v_{i+1}$ to $v_{i+2}$, and also edges $v_0v_1$ and $v_{t-1}v_t$ have different colors -- consequently alternating circuits have even length).

By Euler's famous method one can easily prove the following

\begin{prop}\label{prop:euler}
If $G$ is a \BrbG\ then the edge set can be decomposed to alternating   circuits. If $B$ is a bipartite \BrbG\ then the edge set can be decomposed into alternating cycles.
\end{prop}

If two graphs, $G_1$ and $G_2$ are different realizations of the same degree sequence, then we associate with them the following \BrbG. The vertex set is   $V(G_1)=V(G_2)$ and the edge set is the symmetric difference $E(G_1) \Delta E(G_2)$. An   edge is colored red, if it is in $E(G_1)-E(G_2)$, and it is colored blue, if it is in $E(G_2)-E(G_1)$.

\begin{definition}
  If $G$ is a \BrbG\ then let $\mc_u(G)$ denote the number of the circuits in a
  maximum size $(=$ maximum cardinality$)$ alternating circuit decomposition
  of $G$. If $G_1$ and $G_2$ are two realizations of the same degree sequence
  then let $\mc_u(G_1,G_2)=\mc_u(G)$, where $G$ is the associated \BrbG.
\end{definition}

\begin{definition}
  Let $G_1$ and $G_2$ be two given realizations of $\mathbf{d}.$ Denote by
  $\mathbf{dist}_u(G_1,G_2)$ the length of the shortest swap sequence from
  $G_1$ to $G_2$.
\end{definition}

A pair of vertices $u$ and $v$ will be called a \emph{chord}, if it can hold
an edge. That is for non-bipartite graphs $uv$ is a chord if and only if $u\ne
v$, but for a bipartite graph $B$, $uv$ is a chord if and only if $u\in U(B)$
and $v\in W(B)$ or vice versa. If a circuit $C=v_0\ldots,v_t$ is given and
$v_iv_j$ is a chord, then we will also call the pair $ij$ of subscripts a {\em chord}.

\smallskip

For directed graphs we consider the following definitions: Let $\vec G$ denote
a directed graph (no parallel edges, no loops) with vertex set $X(\vec G) =
\{x_1,x_2, \ldots,x_n \}$ and edge set $E( \vec G)$. We use the
bi-sequence
$$
\bdd(\vec G) =\Big (\left (\a_1,\a_2,\ldots,\a_n \right ), \left (\b_1, \b_2,
  \ldots, \b_n\right ) \Big )
$$
to denote the degree sequence, where $\a_i$ denotes the out-degree of vertex $x_i$ while $\b_i$ denotes its in-degree. A bi-sequence of non-negative integers is called a {\em directed degree sequence} if there exists a  directed graph $\vec G$ such that $\mathbf{(\a,\b)}=\bdd(\vec G)$. In this case we say that $\vec G$ {\em  realizes} our directed degree sequence.

A directed graph $\vec G$ is a {\em \BrbG}, if for every vertex the red in-degree is the same as the blue in-degree, and moreover the red out-degree is the same as the blue out-degree. Thus if $\vec G_1$ and $\vec G_2$ are different realizations of the same directed degree sequence then the associated red-blue graph (defined similarly as for the undirected case) is a \BrbG.

The definition of alternating circuit differs from the one defined for undirected graphs as follows. A circuit $v_0,\ldots,v_t$ in a \BrbG\ $\vec G$ is alternating, if both the colors and the directions alternates (e.g., if $v_{i}v_{i+1}$ is a red directed edge then $v_{i+2}v_{i+1}$ is a blue directed edge).

Again, by Euler's method one can easily prove the following:

\begin{prop}\label{prop:direuler}
If $\vec G$ is a \BrbG\ then the edge set can be decomposed into alternating circuits.
\end{prop}

\begin{definition}
 Assume that $\vec G$ is a directed \BrbG\ and let $\mc_d (\vec G)$ denote the number of   the circuits in a maximum size alternating circuit decomposition of $\vec  G$. If $\vec G_1$ and $\vec G_2$ are two realizations of the same directed degree  sequence then let $\mc_d(\vec G_1,\vec G_2)=\mc_d(\vec G)$, where $\vec G$  is the associated \BrbG.
\end{definition}

For directed graphs we use the old trick, applied already by Gale \cite{G57}: each directed graph $\vec G$ can be represented by a bipartite graph $B({\vec G})$, where each class consists of one copy of every vertex. The edges adjacent to a vertex $u_x$ in class $U$ represent the out-edges from $x$, while the edges adjacent to a vertex $w_x$ in class $W$ represent the in-edges to $x$ (so a directed edge $xy$ is identified with the edge $u_xw_y$). Note that the directed degree sequence of $\vec G$ is the {\bf same} as the bipartite degree sequence of  $B({\vec G})$. If $\vec G$ is a directed \BrbG\ then naturally we get $B({\vec G})$ as a \BrbG, and the alternating circuits of $\vec G$  corresponds to the alternating cycles of $B({\vec G})$. For an alternating circuit $\vec C$ of $\vec G$ we denote the corresponding alternating cycle of $B({\vec G})$ by $C$.

As loops are not allowed in $\vec G$, edges of the form $u_xw_x$ are also forbidden in $B({\vec G})$, so they will be called \emph{non-chords}.

\smallskip

For two graphs $G_1$ and $G_2$ (or bipartite graphs or directed graphs) with the same degree sequence (or bipartite degree sequence or directed degree sequence, resp.) we will use $H'(G_1,G_2)$ for the {\em halved Hamming distance} $\frac{ |E(G_1) \Delta E(G_2)|}{2}$. Note that $H'(G_1,G_2)$ is the same as the number of red (or blue) edges in the associated \BrbG\ $G$.

\section{Undirected degree sequences}\label{sec:undir}

In this Section, we prove equality (\ref{eq:weak}) for  undirected degree sequences.

\begin{lemma}\label{th:even-distance}
Let $C = v_0, v_1, \ldots v_{2t}=v_0$ be an alternating circuit in a \BrbG\ $G,$ in which for some $i<j<2t$, $j-i$ is even and $v_i = v_j$. Then the circuit can be decomposed into two shorter alternating circuits.
\end{lemma}
\begin{proof}
Since both $v_i, v_{i+1}, \ldots v_j$ and $v_j, v_{j+1}, \ldots v_{2t}, v_1, \ldots v_i$ contains  even number of edges, both of them form alternating circuits.
\end{proof}
\begin{definition}
We call an alternating circuit $C = v_0, v_1, \ldots v_{2t}$ {\bf elementary}, if {\rm  (i)} no vertex appears more than twice in it, and if {\rm (ii)} there exists an integer $0\le i< 2t$, such that both vertices $v_i$ and $v_{i+1}$   occur only once in the circuit.
\end{definition}
\begin{lemma}\label{th:elemi}
Let $C_1,\ldots, C_h$ be a maximum size alternating circuit decomposition of a \BrbG\ $G$ $($that is $h=\mc_u(G)).$ Then each circuit is elementary.
\end{lemma}
\begin{proof}
(i) First assume that a circuit $C_z=v_0,\ldots,v_{2t}$ contains the vertex $v$ three times.
Then two of the occurrences have the same parity, and Lemma~\ref{th:even-distance} applies. But this contradicts to the maximality. Therefore any vertex in any circuit of a maximum size decomposition  occurs at most twice, and the two subscripts of
each repeated vertex (within any circuit) have different parities. We call a pair $0\le i<j<2t$ a \emph{non-chord}, if $v_i=v_j$. The {\em length} of a non-chord $ij$ is defined to be $\min(|i-j|,\; 2t-|i-j|)$.  We proved that each index $i$ is a part of at most one non-chord, and the length of any non-chord is odd.

We next prove that non-chords cannot intersect. More precisely, if $0\le i<k<j<\ell<2t$ then, if $ij$ is a non-chord then $k\ell$ is a chord.  Let $C'$ be the following alternating circuit:
$$
v_0,\ldots,v_i=v_j,v_{j-1},v_{j-2}, \ldots, v_k,\ldots v_{i+1},v_i=v_j,v_{j+1},\ldots,v_{\ell} \ldots v_{2t}.
$$
In this new circuit (which consists of the same edges as the original circuit) the new index of vertex $v_k$ is $k'$, where $k-k'$ is odd. Therefore if $k-\ell$ was odd then $k'-\ell$ is even. Thus for this circuit Lemma ~\ref{th:even-distance} applies -- which in turn shows that the original circuit decomposition is not a maximum size one, a contradiction.

(ii) By re-indexing the vertices of the circuit we may assume that $0k$  is the shortest non-chord of $C_z$.  Since $G$ is simple we have $k>2$. Consequently  indices $1$ and $2$ cannot participate in any non-chord otherwise they would induce crossing non-chords.
\end{proof}

From the middle part of the proof one can deduce a much stronger statement. Given a circuit $C= v_0,\ldots,v_{2t}$ we call a vertex {\em unique}, if it appears exactly once in that circuit.

\begin{theorem}\label{th:elemi_strong}
Let $C_1,\ldots, C_h$ be a maximum size alternating circuit decomposition of a \BrbG\ $G$. Then
\begin{enumerate}[{\rm (i)}]
\item  each circuit $C_z=v_0,\ldots,v_{2t}$ contains at least   $2t/3+2$ unique vertices;
\item  the length of each circuit in the   decomposition is at most $(3/2)(n-1)$, consequently  $$\mc_u(G)\ge \left \lceil   \frac{2|E(G)|}{3n}\right \rceil. $$
\end{enumerate}
\end{theorem}
\begin{proof}
(i) Let $\mu$ denote the number of unique vertices of the circuit and let $\nu$ denote the number of non-unique vertices (not indices). Since no vertex appears more than twice in the circuit therefore $2t=\mu+2\nu$ and the number of non-chords is exactly $\nu$.

If $t=2$ then non-chords do not exist so nothing to prove. Assume now that $t>2$ and consider the following planar graph $P$ with $2t$ vertices. First we draw a convex $2t$-gon on the plane with vertices $p_0,\ldots,p_{2t-1}$. Next for each non-chord $ij$ we connect $p_i$ to $p_j$ by a straight line segment. The proof of Lemma~\ref{th:elemi} shows that there are no crossing non-chords therefore this is a planar embedding.

Now we take the planar dual $P^*$ and delete the vertex of the dual corresponding to the infinite face of $P$. We call the resulting graph $T$.

It is easy to see that $T$ is a tree. Indeed, we can argue by contradiction. If $T$ contains a cycle then the original planar graph contains a vertex $v$ within this cycle. But each original vertex of the graph is neighboring to the ocean. Therefore vertex $v$ is also next to the ocean, so the dual vertex $O$  corresponding the ocean must not be outside $C.$ But that would imply that $O$ belongs to the cycle, a contradiction.

The  edges of $T$ correspond to the non-chords, so $|E(T)|=|V(T)|-1=\nu$. The vertices of the tree correspond to the finite faces of $P$. We claim that if $v\in V(T)$ has degree at most $2$ in the tree, then the corresponding face contains at least $2$ unique vertices.  If a face is adjacent to one non-chord (it corresponds to a leaf in $T$) then it has at least two unique vertices since $G$ is simple. Suppose there is a face of $P$ adjacent to two non-chords, $ij$ and  $i'j'$, where we may assume that  $i<i'<j'<j$. As $G$ is simple and the  non-chords have odd length, we can conclude that $i'-i+j-j'\ge 4$, proving the claim.

Let $n_{\le d}$ (and $n_{\ge d}$) denote the number of vertices of $T$  having degree at most $d$ (at least $d$, resp.). We prove by induction on  $|V(T)|$ that  if $|V(T)|>1$ then $n_{\le 1}\ge n_{\ge 3}+2$ (if we delete a  leaf then either none of $n_{\le 1}$ and $n_{\ge 3}$ is changed, or $n_{\le  1}$ decreases by one and $n_{\ge 3}$ decreases by at most one).  Consequently $n_{\le 2}\ge n_{\ge 3}+2$. So
$$
\mu\ge 2n_{\le 2}\ge |V(T)|+2=|E(T)|+3=\nu+3.
$$
As $2t=\mu+2\nu$, we have $2t\le 3\mu -6$, consequently $\mu\ge 2t/3+2$, proving our first statement.

\medskip\noindent (ii) To prove the second statement we only need a simple calculation.
$$
t+(t/3+1)\le \mu/2+\nu+\mu/2=\mu+\nu\le n,
$$
so really $2t\le (3/2)(n-1)$.
\end{proof}

\smallskip\noindent Now we are ready to analyze the minimum size swap sequences. We start with the simplest case:
\begin{lemma}\label{lem:one-circuit1}
Assume that $G_1$ and $G_2$ are two realizations of the same degree  sequence, and $G$ is a \BrbG\ consisting of the edges in $E(G_1)\Delta E(G_2)$. Suppose that $E(G)$ is one alternating elementary circuit  $C$ of length $2t.$ Then there is a swap sequence of length $t-1$   between $G_1$ and $G_2$.
\end{lemma}
\begin{proof}
Let us call $G_1$ the {\em start} and $G_2$ the {\em stop} graph. We apply induction on the size of the symmetric difference $|C|$ of the actual start and stop graphs. We may assume, that $v_0$ occurs exactly once in $C$ and the current $v_0v_1$ edge belongs to the start graph (since this circuit is elementary, due to Lemma \ref{th:elemi}, we can always renumbering the vertices accordingly).

When $t=2$ then the statement is clear, since $C$ is an alternating cycle of length four, so assume now that $t > 2$. Consider the chords $v_0v_1,\; v_1v_2,\; v_2v_3,\; v_3v_0.$ (Let's recall: by definition we have $v_0v_1, v_2v_3 \in E(G_1) \setminus E(G_2)$ and $v_1v_2 \in E(G_2)\setminus E(G_1)$ while $v_3v_0 \not \in E(G_1) \Delta E(G_2).$)

When chord $v_3v_0$ is non-edge in the start (and therefore in the stop) graph, then we can perform the  $v_0v_1, v_2v_3 \Rightarrow v_1v_2, v_3v_0$ swap in the start graph. After this operation the circuit will be shorter by two edges and remains elementary. So we can apply the inductive hypothesis for the new start/stop graph pair. If, however, the chord $v_3v_0$ is an edge both in the start and stop graphs, then we can carry out the  $ v_1v_2, v_3v_0 \Rightarrow   v_0v_1, v_2v_3 $ swap in the stop graph, still maintaining all the necessary properties. So we can proceed with the induction on the new start/stop graph pair.
\end{proof}

\begin{theorem}\label{th:SD-main1}
For all pairs of realizations $G_1, G_2$ of the same degree sequence, we have
\begin{equation}\label{eq:dist1}
\mathbf{dist}_u (G_1,G_2) = H'(G_1,G_2) - \mc_u(G_1,G_2).
\end{equation}
\end{theorem}
\begin{proof}
(i/a) The inequality LHS $\le$ RHS is a simple application of Lemma
\ref{lem:one-circuit1}: take a maximal alternating circuit decomposition
$C_1,...,C_k$ where $k=\mc_u(G_1,G_2)$, and define realizations
$G_1=H_0,H_1,\ldots, H_{k-1}, H_k=G_2$ such that for all $i=0,\ldots, k-1$ realizations $H_i$ and $H_{i+1}$ differ exactly in $C_i.$ Then by Lemma \ref{th:elemi} all circuits are elementary, so the application of Lemma \ref{lem:one-circuit1} for each pair $H_i,H_{i+1}$ proves this inequality.

\medskip\noindent
(i/b) One can find a recursive proof as well. This is based on the following easy observation: assume that the shortest circuit $C_1$ in the previous maximal decomposition has the shortest length among all circuits in all possible maximal circuit decomposition. Then
\begin{lemma}\label{lm:minimal}
There exists no edge in any of the other circuits which would divide $C_1$ into two odd length trails.
\end{lemma}
\medskip\noindent{\em Proof.}
Assume the opposite: the chord $v_1,v_{2\ell}$ of $C_1$ is an edge in $C_2.$ Then this edge together with one of the trails of $C_1$ form a shorter circuit, than $C_1$ while the other trail together with the remaining part of $C_2$ form another alternating circuit. So we constructed another circuit decomposition with the same number of circuits, but with a shorter shortest circuit, a contradiction. (It is still possible that a chord in $C_1$ belongs to another circuit as well -- but this divides $C_1$ into two even-length trails. However this will not cause any problem.) \hfill $\Box_{\ref{lm:minimal}}$

\medskip
Now we can operate as follows: consider the (actual) symmetric difference, find a maximal circuit decomposition with a shortest elementary circuit. Apply the procedure in Lemma \ref{lem:one-circuit1} for this circuit (by Lemma \ref{lm:minimal} we can do it). Repeat the whole process with the new (and smaller) symmetric difference.

\medskip\noindent
(ii) We finish the proof of Theorem \ref{th:SD-main1} by proving that LHS $\ge RHS.$ We realign (\ref{eq:dist1}) into
$$
\mc_u(G_1,G_2) \ge H'(G_1,G_2) - \mathbf{dist}_u(G_1,G_2).
$$
Assume that the sequence $G_1=H_0,H_1,\ldots, H_{k-1}, H_k=G_2$ describe a minimum length realization sequence from $G_1$ to $G_2$ where for each $i=0,\ldots, k-1$, the graphs $H_i$ and $H_{i+1}$ are in swap-distance 1. It is clear that  any consecutive  swap subsequence from $H_i$ to $H_j$ must be also a minimum one. For each $i$ we use the notation
$$\Delta_i :=E_1 \Delta E(H_i) .$$

We are going to construct a circuit decomposition of $E_1\Delta E_2 = \Delta_k$ into  $\ge H'(G_1,G_2) - \mathbf{dist}_u (G_1,G_2)$ alternating circuits. By part (i) it will prove also that the two sides are actually equal (otherwise the swap sequence cannot be minimum). We proceed with induction: we will show that for all $i=0,\ldots,k$ we have
\begin{equation}\label{eq:induct1}
\mc_u(G_1,H_i) \ge H'(G_1, H_i) - \mathbf{dist}_u (G_1,H_i).
\end{equation}
In case of $i=0$ this is clearly true, if $i=k$ then the main statement is proved.  Now we assume (\ref{eq:induct1}) for subscript $i$ and we are going to prove it for $i+1.$ By the hypothesis we know that $ dist_u(G_1, H_i) = i.$ We are going to distinguish cases upon the relations among $E(H_i) \Delta E(H_{i+1})=S$ and $\Delta_i.$

Assume at first that $| S \cap \Delta_i | = 0.$ Then the number of circuits in the decomposition of $\Delta_{i+1}$ is increased by one (comparing to the maximum decomposition of $\Delta_i$), the number of edges is increased by four, finally the number of swaps is increased by one again. Inequality (\ref{eq:induct1}) is maintained.

Now assume that $|S \cap \Delta_i| =\ell > 0.$ Since $S$ is derived from the swap transforming $H_i$ into $H_{i+1}$ therefore the two existing edges among the four chords defining $S$ are edges in $H_i$ and not edges in $H_{i+1}$  and the analogous statement is true for the two missing edges. Therefore the chords in $S \cap \Delta_i$ are in the same states  in $H_0$ and in $H_{i+1}.$ Then
\begin{enumerate}[{\rm (a)}]
\item if $\ell=1$ then this chord does not belong to $\Delta_{i+1}$ therefore the other three chords of $S$ extend the original circuit. Therefore the number of circuits is the same as before, while $|\Delta_{i+1}| = |\Delta_i|+2$ and the number of necessary swaps is increased by one. Inequality (\ref{eq:induct1}) is maintained;
\item if $\ell > 1$ then the $\ell$ common chords can be in at most $\ell$ circuits. It can happen, that some circuits meld into a smaller number of circuits after the swap on $S$ is performed, but this can decrease the number of circuits with at most $\ell -1.$ Furthermore $|\Delta_{i+1}| = |\Delta_i|+4 - 2\ell,$ finally the number of necessary swaps increased by one. Inequality (\ref{eq:induct1}) is maintained.
\end{enumerate}
The proof Theorem \ref{th:SD-main1} is finished.
\end{proof}
\noindent
As it was already mentioned the value $\mc_u$ seems to be not efficiently computable. Therefore Theorem \ref{th:SD-main1} does not help directly to find a shortest swap sequence between two particular realizations. However good upper and lower bounds on this value may be useful. It is clear, however, that these bounds depend not only on the number of the edges (which is  $|E|$)  in one realization  but on the size of the symmetric difference. When $|E|$ is small, say $|E| \le \frac{1}{2} \binom{n}{2}$ then the size of the symmetric difference can be as big as $2|E|.$ If $|E|$ is much higher then the symmetric difference becomes small.

Assume now, that the graphical degree sequence under investigation is $(1,1,1, \ldots,1).$ All realizations are perfect matchings, and if two of them form one alternating Eulerian cycle, then the actual swap-distance, by Theorem \ref{th:SD-main1} is $|E|-1.$ This is just the half of the old estimation. The consequence is that the diameter of the corresponding Markov chain can be as big as $|E|-1.$

Next we give  a general bound on the swap-distance (which is in some sense sharp), and
then we formulate some conjectures.  For a given degree sequence $\bd=\{d_1,d_2, \ldots,d_n\}$ let $m$ denote $(\sum d_i)/2$, the number of edges in any realization, and let $m^*$ denote $\sum\min(d_i, n-d_i)$, an upper bound on the number of edges in a \BrbG\ associated with two realizations $G_1$ and $G_2$.
\begin{theorem}\label{th:undir_bound}
For all pairs of realizations $G_1, G_2$ of the same degree sequence of length
$n$, we have
\begin{eqnarray*}
\mathbf{dist}_u (G_1,G_2) &\le&  H'(G_1,G_2)\cdot \left (1- \frac{4}{3n}\right ) \\ &\le&
m^*\left (\frac{1}{2}-\frac{2}{3n} \right )\le m\left (1-\frac{4}{3n}\right ).
\end{eqnarray*}
\end{theorem}
\begin{proof}
It is a simple calculation using Theorem \ref{th:elemi_strong} and the simple
fact, that $H'(G_1,G_2)\le m^*/2\le m$.
\end{proof}
\begin{conj}
Let $G$ be a \BrbG\ with $n$ vertices and $m$ edges. Then {\rm (i)} there exists an alternating circuit of length at most $3n^2/m$. And {\rm (ii)} $\mc_u(G)\ge m^2/(6n^2)$.
\end{conj}
\noindent Such upper bound would provide a lower bound on the distance, and thus could be useful in practical applications.
\begin{conj}
For a degree sequence
$\bd=\{d_1,d_2,\ldots,d_n\}$ let $m$ again denote $(\sum d_i)/2$, the number of
edges in any realization, and let $m^*$ denote $\sum\min(d_i, n-d_i)$.
Then we conjecture the following statements. The listed inequalities arisen
\begin{enumerate}[{\rm (i)}]
\item $\mathbf{dist}_u (G_1,G_2) \le H'(G_1,G_2)\cdot(1-m/(3n^2))$.
\item $\mathbf{dist}_u (G_1,G_2) \le m^*(1/2-m/(6n^2))$.
\item $\mathbf{dist}_u (G_1,G_2) \le m(1-m/(3n^2))$.
\item $\mathbf{dist}_u (G_1,G_2) \le 5n^2/24$.
\end{enumerate}
\end{conj}

\section{Undirected bipartite degree sequences}\label{sec:bi}

It is easy to see that for bipartite degree sequences Theorem
\ref{th:SD-main1} applies without any changes (note that in the proof we only
used chords of odd length, so they are also chords in the bipartite case).
Even more, since there is no
odd cycle in a bipartite graph, the circuits in the maximal size alternating
circuit decomposition of the symmetric difference of two realizations are
cycles. As a consequence, for two realizations $B_1$ and $B_2$ of a bipartite
degree sequence, we can interpret $\mc_u(B_1,B_2)$ as the maximum number of
cycles in a decomposition into alternating cycles (which always exists)
of the associated \BrbG\ $B$.
However, we think that even for bipartite realizations the determination of
$\mc_u$ might be hard.

Let $\bbd=\big( (a_1, \ldots,a_k ),  (b_1,\ldots ,b_\ell  )\big )$ be a given bipartite degree sequence, we assume $\ell\le k$. Let $n=k+\ell,\; n'=2\ell,\; m=\sum a_i$, and  let $m^*$ denote $2\sum\min(a_i, \ell-a_i)$, an upper bound on the number of edges in a \BrbG\ associated with two realizations $B_1$ and $B_2$. Using that any alternating cycle has length at most $n'$, similarly to Theorem \ref{th:undir_bound}, we get the following.

\begin{theorem}\label{th:bip_bound}
For all pairs of realizations $B_1, B_2$ of the same degree sequence
$\bbd$ we have
\begin{eqnarray*}
\mathbf{dist}_u (B_1,B_2) &\le&  H'(B_1,B_2)\cdot(1-2/n')\\
&\le& m^*(1/2-1/n')\le m(1-2/n'). \qquad \Box
\end{eqnarray*}
\end{theorem}


\section{Directed degree sequences}\label{sec:dir}

In this section we discuss directed degree sequences.  We will apply the machinery of Section \ref{sec:bi} to solve the directed degree sequence problem, using the bipartite graph $B({\vec G})$ defined in Section \ref{sec:def}. However doing so we may face a serious problem: since no loop is allowed in $\vec G$, we cannot use edges of form $u_xw_x$ in the process. Recall that these pairs are called non-chords.  So at first we are going to analyze the alternating cycles we have to handle along the process.

Let $\vec G$ be a directed \BrbG\ associated with two realizations $\vec G_1$ and $\vec G_2$ of the same directed degree sequence, let $B=B({\vec G})$ be the corresponding bipartite \BrbG, and let $\vec C$ be an alternating circuit in $\vec G$ (recall, that $C$ denotes the corresponding alternating cycle in $B$). In this section we will mainly use the terminology about the bipartite representation $B$, but, where it is interesting, we
remark in italics and in parenthesis the corresponding notions in the original directed graph.

\medskip\noindent {\bf Case 1}: Let us start with the case, when there exists a vertex $u_x$ in the cycle $C$ such that $w_x$ is not contained in $C$ (of course, by symmetry, the case when $w_x$ is in $C$, but $u_x$ is not, can be handled by the same way).  ({\it This is equivalent to saying that circuit   $\vec C$ contains $x$ only once.}) We can work with $u_x$ at each step of the process described in Lemma \ref{lem:one-circuit1}: we take the trail of
length 3 starting at $u_x$, and interchange the start and stop graphs if the first edge belongs to the stop graph (we will do this step again and again as a routine, observe that if we are given a swap sequence from one graph to another graph, then the reverse sequence transforms the second graph to the first one). And at every step the vertex $u_x$ remain in the cycle and $w_x$ will not become a vertex of the cycle.

\medskip\noindent{\bf Case 2}: Next assume that for each vertex $u_x\in C$ we also have $w_x \in C$, but also assume that we have a vertex $u_x$ in $C$, such that the trail of length 3 (along the ordering of  the cycle) starting at $u_x$ does not end at $w_x.$ Then we can use vertex $u_x$ at the process as before. After the first swap, there will be two vertices which will occur without their non-chord pairs in the new cycle. So we are back to Case 1.

\medskip\noindent{\bf Case 3}: Finally assume that neither Case 1 nor Case 2 applies.  We can handle this case as follows. Assume first that $C$ is long enough, that is $|C| \ge 8.$ As every vertex participates in one non-chord, one of the two different vertices that are of distance 3 (along the cycle) from any fixed vertex $u_x$, must differ from $w_x$. So we are back to Case 2 after reversing the description of $C$ and possibly interchanging the start
and stop graphs.

\noindent From now on we will call these ``usual''
swaps as {\bf $C_4$-swaps}.

\smallskip
When our cycle is of length 6, then no such trick works. ({\it In $\vec G_1$ we have an oriented triangle and $\vec G_2$ is identical with $\vec G_1$ except that it contains the other orientation of the same triangle. Then we have to use a new type of swap: we exchange the first oriented triangle to the second one.}) This means that in the bipartite graph we swap a $C_6$ with 3 non-chords in one step.  For obvious reason we will call this new swap as {\bf triangular   $C_6$-swap} and the cycle itself is called a {\bf triangular $C_6$ cycle}.

\begin{remark}\label{lm:no-trian}
We can describe now the type 2 triple-swaps mentioned in Section \ref{sec:intro} and introduced in {\rm \cite{KW73, K09}}: they are simply  non-triangular $C_6$-swaps that can be implemented by two $C_4$-swaps.
\end{remark}

\begin{lemma}\label{lem:nonew}
If $C$ is a cycle in the decomposition that is not a triangular $C_6$ cycle, then we can always perform the next swap without producing a new triangular $C_6$ cycle.
\end{lemma}
\begin{proof}
This is clearly the case when we are in Case 1.  If we are in Case 3 or in   Case 2, then $|C|\ge 8$ and after the first swap two neighboring vertices  are deleted from the cycle, resulting that we are back to Case 1 (two non-chords disappear).
\end{proof}

\smallskip\noindent It is important to recognize that sometimes triangular $C_6$-swaps are absolutely necessary: for example let $n\ge 2$ be an integer and consider the following $n+1$-element directed degree sequence: $\bdd = \big ( (n,n,\ldots,n, n-1, n-1, n-1); (n,n,\ldots,n, n-1, n-1, n-1)\big ).$ It is clear that there are exactly two different realizations of this directed degree sequence: namely this is a complete directed graph on $n+1$ vertices minus one oriented triangle -- and this oriented triangle can be of two different kinds. And for these realizations there are exactly one possible swap: the triangular $C_6$-swap on that six vertices of the $B(\vec G_i)$ realizations. (The simplest such example is $((1,1,1), (1,1,1))$.)

With these observations we just proved, that any realization of a directed degree sequence can be transferred to any other realization of the same degree sequence using only $C_4$- and triangular $C_6$-swaps. Therefore from now  on -- opposing papers \cite{KW73} and \cite{K09} -- we allow only these two types of swaps, while three-edge swaps of type 2 are not allowed anymore.

\bigskip\noindent Next we are going to analyze the structure of the triangular $C_6$ cycles in a maximal cycle decomposition with minimal number of triangular $C_6$ cycles. ({\it Let us start with an example $($see Figure \ref{fig:real}$)$: in this  directed degree sequence the realizations consists of two oriented triangles, sharing one vertex.})

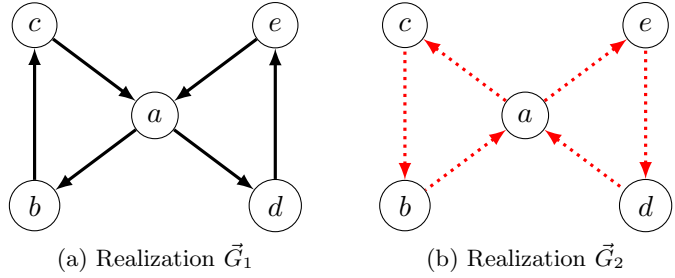
\begin{figure}[h!]
\center{Let the vertices be $X=\{a,b,c,d,e\}$ and $\bdd=\big ( (2,1,1,1,1);(2,1,1,1,1)\big )$ \medskip
 \subfloat[Realization $\vec G_1$]
    {
\begin{tikzpicture}[scale=0.4]
\begin{scope}[>=latex]
\node at (1,0) [shape=circle,draw] (p1) {$b$};
\node at (1,6) [shape=circle,draw] (p2) {$c$};
\node at (5,3) [shape=circle,draw] (p3) {$a$};
\node at (9,0) [shape=circle,draw] (p4) {$d$};
\node at (9,6) [shape=circle,draw] (p5) {$e$};
\draw [very thick,->] (p1) -- (p2);
\draw [very thick,->] (p2) -- (p3);
\draw [very thick,->] (p3) -- (p1);
\draw [very thick,->] (p3) -- (p4);
\draw [very thick,->] (p4) -- (p5);
\draw [very thick,->] (p5) -- (p3);
\end{scope}
\end{tikzpicture}
    }
    \qquad
\subfloat[Realization $\vec G_2$]
    {
\begin{tikzpicture}[scale=0.4]
\begin{scope}[>=latex]
\node at (1,0) [shape=circle,draw] (p1) {$b$};
\node at (1,6) [shape=circle,draw] (p2) {$c$};
\node at (5,3) [shape=circle,draw] (p3) {$a$};
\node at (9,0) [shape=circle,draw] (p4) {$d$};
\node at (9,6) [shape=circle,draw] (p5) {$e$};
\draw [very thick,dotted,red,<-] (p1) -- (p2);
\draw [very thick,dotted,red,<-] (p2) -- (p3);
\draw [very thick,dotted,red,<-] (p3) -- (p1);
\draw [very thick,dotted,red,<-] (p3) -- (p4);
\draw [very thick,dotted,red,<-] (p4) -- (p5);
\draw [very thick,dotted,red,<-] (p5) -- (p3);
\end{scope}
\end{tikzpicture}
    }
\caption{Two realizations}\label{fig:real}
    }

\end{figure}
\noindent Figure \ref{fig:bipart} shows the bipartite representation of the symmetric difference of the corresponding $B_1$ and $B_2$. It is easy to see that there are two possible cycle decompositions of this symmetric difference: one consists of two triangular $C_6$ cycles, but the other one contains none (see Figure \ref{fig:decomp}).
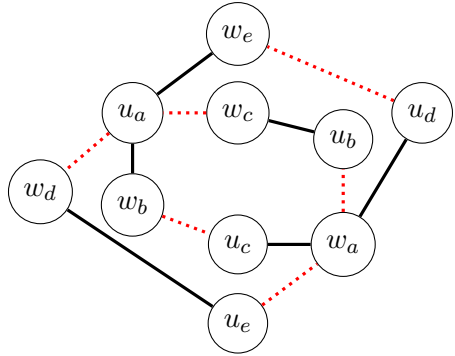
\begin{figure}[h!]
\center{
\begin{tikzpicture}[scale=0.35]
\begin{scope}[>=latex]
\node at (1,6) [shape=circle,draw] (p1) {$u_a$};
\node at (1,2.5) [shape=circle,draw] (p2) {$w_b$};
\node at (5,1) [shape=circle,draw] (p3) {$u_c$};
\node at (9,1) [shape=circle,draw] (p4) {$w_a$};
\node at (9,5) [shape=circle,draw] (p5) {$u_b$};
\node at (5,6) [shape=circle,draw] (p6) {$w_c$};

\node at (-2.5,3) [shape=circle,draw] (p7) {$w_d$};
\node at (5,-2) [shape=circle,draw] (p8) {$u_e$};
\node at (12,6) [shape=circle,draw] (p9) {$u_d$};
\node at (5,9) [shape=circle,draw] (p10) {$w_e$};
\draw [very thick] (p1) -- (p2);
\draw [very thick,dotted,red] (p2) -- (p3);
\draw [very thick] (p3) -- (p4);
\draw [very thick,dotted,red] (p4) -- (p5);
\draw [very thick] (p5) -- (p6);
\draw [very thick,dotted,red] (p6) -- (p1);
\draw [very thick,dotted,red] (p1) -- (p7);
\draw [very thick] (p7) -- (p8);
\draw [very thick,dotted,red] (p8) -- (p4);
\draw [very thick] (p4) -- (p9);
\draw [very thick,dotted,red] (p9) -- (p10);
\draw [very thick] (p10) -- (p1);
\end{scope}
\end{tikzpicture}
\caption{The bipartite representation of the symmetric difference}\label{fig:bipart}
    }
\end{figure}

It is a fortune that this is the typical behavior. We say that two cycles in the decomposition of the bipartite representation are {\bf kissing}, if there exists a vertex $x\in X$ such that both alternating cycles in the decomposition contain both $u_x$ and $w_x$. If one or both kissing cycles are triangular $C_6$ cycles then we can transform these two cycles into a new decomposition, without any triangular $C_6$. For that end we consider the four trails defined by $u_x$ and $w_x$ and pair them up in the right way.
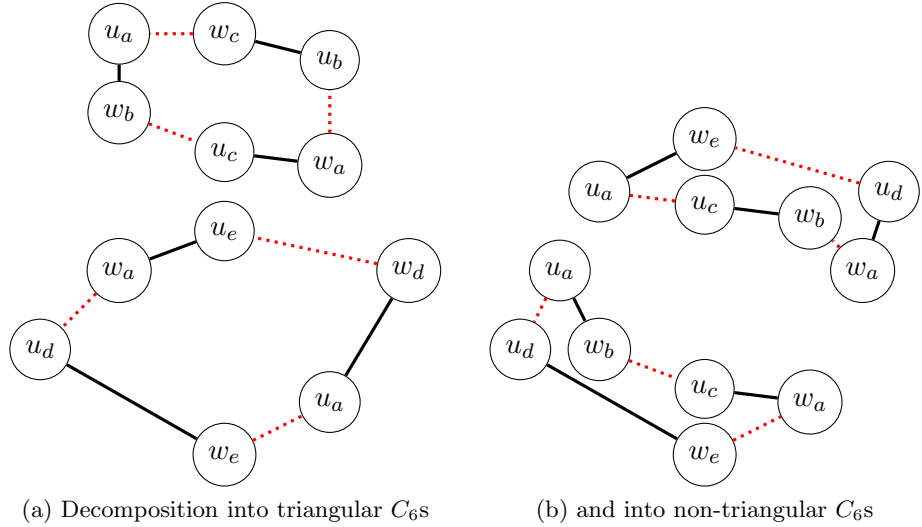
\begin{figure}[h!]
\center{\subfloat[Decomposition into triangular $C_6$s]
    {
\begin{tikzpicture}[scale=0.35]
\node at (1,15) [shape=circle,draw] (p1) {$u_a$};
\node at (1,12) [shape=circle,draw] (p2) {$w_b$};
\node at (5,10.5) [shape=circle,draw] (p3) {$u_c$};
\node at (9,10) [shape=circle,draw] (p4) {$w_a$};
\node at (9,14) [shape=circle,draw] (p5) {$u_b$};
\node at (5,15) [shape=circle,draw] (p6) {$w_c$};

\node at (1,6) [shape=circle,draw] (p11) {$w_a$};
\node at (9,1) [shape=circle,draw] (p12) {$u_a$};
\node at (-2,3) [shape=circle,draw] (p7) {$u_d$};
\node at (5,-1) [shape=circle,draw] (p8) {$w_e$};
\node at (12,6) [shape=circle,draw] (p9) {$w_d$};
\node at (5,7.5) [shape=circle,draw] (p10) {$u_e$};
\draw [very thick] (p1) -- (p2);
\draw [very thick,dotted,red] (p2) -- (p3);
\draw [very thick] (p3) -- (p4);
\draw [very thick,dotted,red] (p4) -- (p5);
\draw [very thick] (p5) -- (p6);
\draw [very thick,dotted,red] (p6) -- (p1);

\draw [very thick,dotted,red] (p11) -- (p7);
\draw [very thick] (p7) -- (p8);
\draw [very thick,dotted,red] (p8) -- (p12);
\draw [very thick] (p12) -- (p9);
\draw [very thick,dotted,red] (p9) -- (p10);
\draw [very thick] (p10) -- (p11);
\end{tikzpicture}
    } \quad
\subfloat[and into non-triangular $C_6$s]
    {
\begin{tikzpicture}[scale=0.35]
\node at (-.5,6) [shape=circle,draw] (p1) {$u_a$};
\node at (1,3) [shape=circle,draw] (p2) {$w_b$};
\node at (5,1.5) [shape=circle,draw] (p3) {$u_c$};
\node at (9,1) [shape=circle,draw] (p4) {$w_a$};
\node at (-2,3) [shape=circle,draw] (p7) {$u_d$};
\node at (5,-1) [shape=circle,draw] (p8) {$w_e$};

\node at (1,9) [shape=circle,draw] (p11) {$u_a$};
\node at (9,8) [shape=circle,draw] (p5) {$w_b$};
\node at (5,8.5) [shape=circle,draw] (p6) {$u_c$};
\node at (11,6) [shape=circle,draw] (p12) {$w_a$};
\node at (12,9) [shape=circle,draw] (p9) {$u_d$};
\node at (5,11) [shape=circle,draw] (p10) {$w_e$};
\draw [very thick] (p1) -- (p2);
\draw [very thick,dotted,red] (p2) -- (p3);
\draw [very thick] (p3) -- (p4);
\draw [very thick,dotted,red] (p1) -- (p7);
\draw [very thick] (p7) -- (p8);
\draw [very thick,dotted,red] (p8) -- (p4);

\draw [very thick,dotted,red] (p12) -- (p5);
\draw [very thick] (p5) -- (p6);
\draw [very thick,dotted,red] (p6) -- (p11);

\draw [very thick] (p12) -- (p9);
\draw [very thick,dotted,red] (p9) -- (p10);
\draw [very thick] (p10) -- (p11);
\end{tikzpicture}
    }
\caption{Two possible cycle decompositions}\label{fig:decomp}
    }
\end{figure}

\medskip \noindent With this observation we just proved the following structural property:
\begin{lemma}\label{lm:no-touch}
Assume that the alternating cycle decomposition $\mathcal{C}$ of the symmetric difference of $B(\vec G_1)$ and $B(\vec G_2)$ is a maximal one with minimum number of triangular $C_6$ cycles. Then no triangular $C_6$ cycle kisses any other cycle.
\end{lemma}

\bigskip\noindent
We are ready now to define the swap-distance of two arbitrary realizations of the same directed degree sequence. We consider a {\bf weighted swap-distance}: an ordinary $C_4$-swap weighs one, but a triangular $C_6$-swap weighs two. (This convention is well supported with the fact that two kissing triangular cycles can be transformed into two ordinary, length 6 cycles, each of them transformable using two $C_4$-swaps.) So $\mathbf{dist}_d (\vec G_1,\vec G_2)$ denotes the minimum total weight of a swap sequence transforming $\vec G_1$ to $\vec G_2$. The definition of  $ \mc_d(\vec G_1,\vec G_2)$ is  analogous to the undirected case: this is the possible maximum number of directed cycles in an alternating  directed cycle decomposition of the symmetric difference of the edge sets.   Using these definitions we have the following result on the minimum directed swap-distance:
\begin{theorem}\label{th:directed}
Let $\bdd$ be a directed degree sequence with realizations $\vec G_1$ and $\vec G_2$. Then
\begin{equation}\label{eq:d-dist}
\mathbf{dist}_d (\vec G_1,\vec G_2) = H'(\vec G_1,\vec G_2) - \mc_d(\vec
G_1,\vec G_2).
\end{equation}
\end{theorem}
\begin{proof}
We can prove that the LHS is at most as big as  the RHS by recalling the proof
from Theorem \ref{th:SD-main1} for bipartite graphs taking into consideration
the previous observations. And, by Lemma \ref{lem:nonew} we do not create any
new triangular $C_6$ cycle.

Consider now the equivalent form of inequality (\ref{eq:induct1}). Fix a
suitable swap sequence of minimal weighted length and assume that this
contains the smallest possible number of triangular $C_6$-swaps among all such
sequences. Denote by $B(\vec G_1)= H_0, H_1,\ldots, H_{k-1}, H_k=B(\vec G_2)$
the bipartite graph sequence which consists of the consecutive realizations
generated by this swap sequence (then for each $i=0,\ldots, k-1$, the graphs
$H_i$ and $H_{i+1}$ are in swap-distance 1 or 2, depending whether the swap
was a simple $C_4$-swap or a triangular $C_6$-swap). It is clear that
\begin{enumerate}[{\rm (A)}]
\item   {\it any consecutive  swap subsequence from $H_i$ to $H_j$ must be also a minimum one, furthermore it must contain the smallest possible number of triangular $C_6$-swaps among all such subsequences.}
\end{enumerate}
 For each $i$ we use the following notation
$$\Delta_i :=E(H_0) \Delta E(H_i) .$$
We will now revisit the proof of Theorem \ref{th:SD-main1} (ii). We will try to mimic that proof. Whenever the swap under study is a triangular $C_6$-swap, then it cannot share a non-chord with any earlier generated cycle. (This comes form a simple straightforward  generalization of Lemma \ref{lm:no-touch}.)

Whenever the swap under investigation is a regular $C_4$-swap, then we can proceed as in the proof of Theorem \ref{th:SD-main1} (ii).  And this concludes the proof of Theorem \ref{th:directed}.
\end{proof}

\medskip\noindent We can strengthen LaMar's recent result (\cite{lamar}):
\begin{theorem}\label{lm:lamar}
If we modify the definition of weighted swap-distance between two realizations of a directed degree sequence such that any length 6 circuit can be swapped in one step, and their weights are 2, then Theorem \ref{th:directed} still holds, and there exists a shortest swap sequence between any two realizations, $\vec G_1$ and $\vec G_2$, of the same directed degree sequence which contains only $C_4$-swaps and triangular $C_6$-swaps.
\end{theorem}
\noindent The novelty here is that there always exists shortest swap sequence which is conform with LaMar's idea.
\begin{proof}
Any length 6 circuit  which is not a triangular $C_6$ cycle falls in the above-discussed Case 1 or 2, and can be transformed with two $C_4$-swaps, whose summed weight is also 2.
\end{proof}

Theorem \ref{th:bip_bound} transforms to the following.  Let us given a directed degree
sequence $\bdd=\big( (d_1^+, \ldots,d_n^+ ), (d_1^-,\ldots ,d_n^- )\big )$.
Let
$$
m=\sum d_i^+ \quad\mbox{and}\quad m^* = \sum_i \left [ \min \left (d_i^+,
n-d_i^+\right )+\min\left (d_i^-, n-d_i^- \right ) \right ]
$$
where $m^*$ is an upper bound on the number of edges in a \BrbG\ associated with two realizations $\vec G_1$ and $\vec G_2$. Then
\begin{theorem}\label{th:dir_bound}
For all pairs of realizations $\vec G_1, \vec G_2$ of the same degree sequence
$\bdd$ we have
\begin{eqnarray*}
\mathbf{dist}_d (\vec G_1,\vec G_2) & \le & H'(\vec G_1,\vec G_2)\cdot \left (1- \frac{1}{n} \right ) \\
& \le & m^* \left (\frac{1}{2}- \frac{1}{2n} \right )\le m \left (1- \frac{1}{n}\right ).
\end{eqnarray*}
\end{theorem}

\bigskip\noindent
To finish our paper, recall that Greenhill proved in \cite[Lemma 2.2]{green} that  in case of regular degree sequences any two directed realizations can be transformed to each other using $C_4$-swaps only. (In this case she calls these $C_4$-swaps {\em switches}.) A similar notion was studied by Berger and M\"uller-Hannemann (see \cite{BMH}). However consider the following example: let $\bdd$ be a two-regular directed degree sequence with six vertices. In Figure \ref{fig:green} we show two realizations:
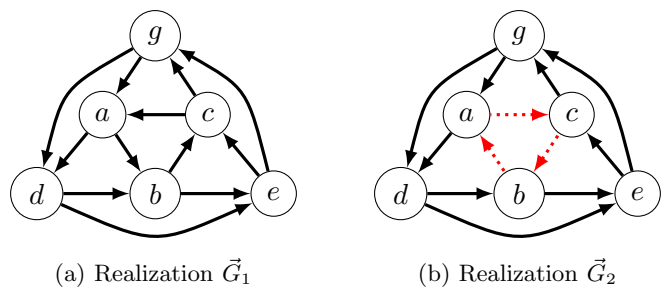
\begin{figure}[h!]
\center{\subfloat[Realization $\vec G_1$]
    {
\begin{tikzpicture}[scale=0.35]
\begin{scope}[>=latex]
\node at (2.5,3) [shape=circle,draw] (p1) {$a$};
\node at (4.5,0) [shape=circle,draw] (p2) {$b$};
\node at (6.5,3) [shape=circle,draw] (p3) {$c$};
\node at (0,0) [shape=circle,draw] (p4) {$d$};
\node at (9,0) [shape=circle,draw] (p5) {$e$};
\node at (4.5,6) [shape=circle,draw] (p6) {$g$};
\draw [very thick,->] (p1) -- (p2);
\draw [very thick,->] (p2) -- (p3);
\draw [very thick,->] (p3) -- (p1);
\draw [very thick,->] (p1) -- (p4);
\draw [very thick,->] (p4) -- (p2);
\draw [very thick,->] (p2) -- (p5);
\draw [very thick,->] (p5) -- (p3);
\draw [very thick,->] (p3) -- (p6);
\draw [very thick,->] (p6) -- (p1);
\draw [very thick,->] (p4) .. controls (4.5,-2) .. (p5);
\draw [very thick,->] (p5) .. controls (8,4) .. (p6);
\draw [very thick,->] (p6) .. controls (1,4) .. (p4);
\end{scope}
\end{tikzpicture}
    } \qquad
\subfloat[Realization $\vec G_2$]
    {
\begin{tikzpicture}[scale=0.35]
\begin{scope}[>=latex]
\node at (2.5,3) [shape=circle,draw] (p1) {$a$};
\node at (4.5,0) [shape=circle,draw] (p2) {$b$};
\node at (6.5,3) [shape=circle,draw] (p3) {$c$};
\node at (0,0) [shape=circle,draw] (p4) {$d$};
\node at (9,0) [shape=circle,draw] (p5) {$e$};
\node at (4.5,6) [shape=circle,draw] (p6) {$g$};
\draw [very thick,dotted,red,<-] (p1) -- (p2);
\draw [very thick,dotted,red,<-] (p2) -- (p3);
\draw [very thick,dotted,red,<-] (p3) -- (p1);
\draw [very thick,->] (p1) -- (p4);
\draw [very thick,->] (p4) -- (p2);
\draw [very thick,->] (p2) -- (p5);
\draw [very thick,->] (p5) -- (p3);
\draw [very thick,->] (p3) -- (p6);
\draw [very thick,->] (p6) -- (p1);
\draw [very thick,->] (p4) .. controls (4.5,-2) .. (p5);
\draw [very thick,->] (p5) .. controls (8,4) .. (p6);
\draw [very thick,->] (p6) .. controls (1,4) .. (p4);
\end{scope}
\end{tikzpicture}
    }
    }
\caption{Two realizations with a triangular $C_6$ as symmetric difference}
\label{fig:green}
\end{figure}
The symmetric difference of these two realizations is one triangular $C_6$ cycle. Therefore the swap sequence generated by Greenhill cannot be a minimal one. (Of course in her application this was never a requirement: she uses the above mentioned result successfully to prove rapid mixing time of the sampling algorithm for regular directed bipartite graphs.)

\end{document}